\documentclass{amsart}
\numberwithin{equation}{section}

  \usepackage{verbatim}
\usepackage{mathrsfs}
\usepackage{latexsym}
\usepackage{epsfig}
\usepackage{amsmath}
\usepackage{bbm}
\usepackage{graphics}
\usepackage{amssymb}
\usepackage{amsthm}
\usepackage{amsrefs}
\usepackage{amsopn}
\usepackage{amscd}
\usepackage[all,knot]{xy}
\xyoption{all}
\usepackage{rotating}
\usepackage{hyperref}
\usepackage[usenames,dvipsnames]{color}

\theoremstyle{plain}
\newtheorem{thm}{Theorem}[section]
\newtheorem{lem}[thm]{Lemma}
\newtheorem{prop}[thm]{Proposition}
\newtheorem{cor}[thm]{Corollary}

\newtheorem*{thm*}{Theorem}
\newtheorem*{lem*}{Lemma}
\newtheorem*{prop*}{Proposition}
\newtheorem*{cor*}{Corollary}

\theoremstyle{definition}
\newtheorem{defn}[thm]{Definition}
\newtheorem*{defn*}{Definition}

\newtheorem{ex}[thm]{Example}
{}
\newtheorem{rem}[thm]{Remark}
\newtheorem*{rem*}{Remark}

\newtheorem{notation}[thm]{Notation}{}
{}
{}
\newtheorem*{ack}{Acknowledgements}{}

\theoremstyle{remark}
{}
{}
{}

\def\cie{\subseteq}

\def\iso{\cong}

\def\intersec{\cap}

\def\to{\longrightarrow}

\def\int{\mathbb{Z}}

\def\L{\mathbf{L}}

\def\l{\lambda}

\def\S{\Sigma}

\def\D{\mathsf{D}}

\def\mcS{\mathcal{S}}
\def\mcT{\mathcal{T}}
\def\mcP{\mathcal{P}}
\def\mcQ{\mathcal{Q}}
\def\GGamma{\mathit{\Gamma}}

\def\CRing{\mathsf{CRing}}

\DeclareMathOperator{\Spec}{Spec}

\DeclareMathOperator{\Spc}{Spc}
\DeclareMathOperator{\Supp}{Supp}
\DeclareMathOperator{\supp}{supp}
\DeclareMathOperator{\supph}{supph}

\DeclareMathOperator{\im}{im}
\DeclareMathOperator{\Hom}{Hom}

\DeclareMathOperator{\Vis}{Vis}

\title{Derived categories of absolutely flat rings}
\author{Greg Stevenson}
\address{Universit\"at Bielefeld, Fakult\"at f\"ur Mathematik, BIREP Gruppe, Postfach 10\,01\,31, 33501 Bielefeld, Germany.}
\email{gstevens@math.uni-bielefeld.de}

\thanks{
This research was partly supported by a fellowship from the Alexander von Humboldt Foundation.}
\begin{document}



\begin{abstract}
\noindent Let $S$ be a commutative ring with topologically noetherian spectrum and let $R$ be the absolutely flat approximation of $S$. We prove that subsets of the spectrum of $R$ parametrise the localising subcategories of $\D(R)$. Moreover, we prove the telescope conjecture holds for $\D(R)$. We also consider unbounded derived categories of absolutely flat rings which are not semi-artinian and exhibit an example of a cohomological Bousfield class that is not a Bousfield class.
\end{abstract}

\maketitle

\tableofcontents

\section{Introduction}
Given a commutative noetherian ring $S$ the structure of the unbounded derived category $\D(S)$ and its full subcategory of compact objects $\D^\mathrm{perf}(S)$ is very well understood. By a result of Hopkins and Neeman \cite{NeeChro} the lattice of thick subcategories of $\D^\mathrm{perf}(S)$ is isomorphic to the lattice of specialisation closed subsets of $\Spec S$. Neeman proves in \cite{NeeChro}*{Theorem~2.8} that when one passes to $\D(S)$ this extends to a lattice isomorphism between the subsets of $\Spec S$ and localising subcategories of $\D(S)$.

Now suppose $S$ is commutative but not noetherian. Then, by work of Thomason \cite{Thomclass}*{Theorem~3.15}, the classification of thick subcategories of $\D^\mathrm{perf}(S)$ is still valid; there is a lattice isomorphism between thick subcategories of $\D^\mathrm{perf}(S)$ and Thomason subsets of $\Spec S$ i.e., those subsets which can be written as unions of closed subsets with quasi-compact complements. However, we are almost completely ignorant concerning the structure of $\D(S)$; we do not even know if the collection of localising subcategories forms a set rather than a proper class. Some partial results, indicating the possible complexity of the lattice of localising subcategories, have been obtained by Neeman \cite{NeeOddball} and then in later work of Dywer and Palmieri \cite{Dwyer_Palmieri_BL} for non-noetherian truncated polynomial rings.

In this work we consider the situation for a commutative absolutely flat ring $R$ i.e., $R$ is commutative, reduced, and zero dimensional. We observe two starkly contrasting ways in which the derived category of such a ring can behave. If $R$ is not semi-artinian we show in Theorem~\ref{prop_gen_fail} that the residue fields of $R$ do not generate the derived category and as a corollary we exhibit a localising subcategory which cannot be realised as the kernel of tensoring by any object of $\D(R)$ (Corollary~\ref{cor_not_bousfield}). This gives a counterexample to the analogue of a conjecture of Hovey and Palmieri \cite{HP_Bousfield}*{Conjecture~9.1} (concerning the stable homotopy category) in the setting of derived categories of rings.

On the other hand, suppose $S$ is a commutative ring with topologically noetherian spectrum. The absolutely flat approximation of $S$ is an absolutely flat ring $S^\mathrm{abs}$ together with a morphism $S\to S^\mathrm{abs}$ through which all other morphisms from $S$ to an absolutely flat ring factor. We prove in Theorem~\ref{thm_class} that Neeman's classification is valid for $\D(S^\mathrm{abs})$: there are a set of localising subcategories and the lattice they form is isomorphic to the powerset of $\Spec S^\mathrm{abs}$. Furthermore, every localising subcategory of $\D(S^\mathrm{abs})$ is the kernel of tensoring with a module and the telescope conjecture holds (Theorem~\ref{thm_telescope}).

Now let us very briefly sketch the contents of the paper. Section~\ref{sec_nonsense} contains abstract results on tensor triangular geometry and support theory. More specifically, after some brief recollections on supports, we prove some results concerning the behaviour of supports under base change (Section~\ref{sec_basechange}) and give a topological restriction on the supports of quotients by smashing subcategories (Section~\ref{sec_smashing}); if the reader is mainly interested in the results concerning derived categories she can safely skip or refer back to this section. Section~\ref{sec_af_prelims} contains some preliminary material on absolutely flat rings and absolutely flat approximations. The main results concerning derived categories are proved in Section~\ref{sec_dercat}.

\begin{ack}
I am grateful to Jesse Burke and Mike Prest for helpful comments relating to this work. Many thanks are also due to Ivo Dell'Ambrogio for several helpful comments on a preliminary version of this article.
\end{ack}

\section{Some tensor triangular abstract nonsense}\label{sec_nonsense}

\subsection{Preliminaries}
We begin by briefly discussing some notions and notation that we will use throughout (we assume some familiarity with the subject matter; for further details and definitions the interested reader should consult \cite{BaRickard}). Let $\mathcal{K}$ be an essentially small rigid tensor triangulated category. We say a subset $\mathcal{V}$ of $\Spc \mathcal{K}$ is \emph{Thomason} if $\mathcal{V}$ can be written as a union of closed subsets of $\Spc \mathcal{K}$ each of which has quasi-compact complement. There is, by \cite{BaSpec}*{Theorem~4.10}, a bijection between Thomason subsets of $\Spc \mathcal{K}$ and thick tensor ideals of $\mathcal{K}$ sending a Thomason subset $\mathcal{V}$ to the ideal $\mathcal{K}_\mathcal{V}$ of all objects whose support is contained in $\mathcal{V}$.

Now suppose $\mcS$ is a rigidly-compactly generated tensor triangulated category. We denote, for $\mathcal{V}\cie \Spc \mcS^c$ a Thomason subset, by $\GGamma_\mathcal{V}\mcS$ the localising ideal generated by $\mcS^c_\mathcal{V}$. This is a smashing subcategory of $\mcS$ and the corresponding Rickard idempotents, giving rise to the acyclisation and localisation functors with respect to $\GGamma_\mathcal{V}\mcS$ will be denoted by $\GGamma_\mathcal{V}\mathbf{1}$ and $L_\mathcal{V}\mathbf{1}$ respectively.

Given $x\in \Spc \mcS^c$ we, as usual, set $\mathcal{V}(x) = \overline{\{x\}}$ and
\begin{displaymath}
\mathcal{Z}(x) = \{y \; \vert \; x\notin \mathcal{V}(y)\}.
\end{displaymath}

A point $x$ of $\Spc \mcS^c$ is said to be \emph{visible} if there exist Thomason subsets $\mathcal{V}, \mathcal{W}$ of $\Spc \mcS^c$ such that $\mathcal{V}\setminus \mathcal{W} = \{x\}$ (this is not the same definition as given in \cite{BaRickard} cf.\ Remark \ref{rem_vis_defn} for further details). In this case we set
\begin{displaymath}
\GGamma_x\mathbf{1} = \GGamma_\mathcal{V}\mathbf{1} \otimes L_\mathcal{W}\mathbf{1}
\end{displaymath}
and recall from \cite{BaRickard}*{Corollary~7.5} that up to isomorphism $\GGamma_x\mathbf{1}$ does not depend on the choice of $\mathcal{V}$ and $\mathcal{W}$.

\begin{rem}\label{rem_goodchoice}
With notation as above we can always choose $\mathcal{W} = \mathcal{Z}(x)$. Indeed, $\mathcal{Z}(x)$ is always Thomason (by virtue of being the support of the prime ideal of $\mcS^c$ corresponding to the point $x$) and $x\notin \mathcal{W}$ combined with the fact that $\mathcal{W}$ is specialisation closed implies $\mathcal{W} \cie \mathcal{Z}(x)$.
\end{rem}

Denoting by $\Vis \mcS^c$ the set of visible points in $\Spc \mcS^c$ we define the \emph{big support} of $X\in \mcS$ to be
\begin{displaymath}
\Supp X = \{x\in \Vis \mcS^c \; \vert \; \GGamma_x\mathbf{1} \otimes X \neq 0\}.
\end{displaymath}

As the big support is a subset of $\Vis \mcS^c$ which can be properly contained in $\Spc \mcS^c$ we cannot in general have, for $s\in \mcS^c$, an equality between $\supp s$ and $\Supp s$. However, no information is lost in the sense that the big support of a compact object determines the usual support. Although we will not use this fact we will give a proof.

\begin{lem}
Let $s$ be a compact object of $\mcS$. Then there is an equality
\begin{displaymath}
\Supp s = \supp s \intersec \Vis \mcS^c.
\end{displaymath}
\end{lem}
\begin{proof}
Using Remark~\ref{rem_goodchoice} one can mimic the proof of \cite{BaRickard}*{Proposition~7.17}.
\end{proof}

\begin{prop}
Let $s$ be a compact object of $\mcS$. Then there is an equality
\begin{displaymath}
\overline{\Supp s} = \supp s.
\end{displaymath}
\end{prop}
\begin{proof}
Let $x$ be a point of $\supp s \setminus \Vis \mcS^c$. By the lemma it is sufficient to show there is a visible point of $\supp s$ specialising to $x$. As $\supp s$ is a closed subspace of $\Spc \mcS^c$ it is a spectral space in its own right. Thus $x$ is a specialisation of a point $y\in \supp s$ which is minimal with respect to specialisation in $\supp s$. To complete the argument it is sufficient to show that $y$ is visible. This is the case since $\supp s$ is Thomason and $\supp s \setminus \mathcal{Z}(y) = \{y\}$ by minimality of $y$ with respect to specialisation in $\supp s$.
\end{proof}

Thus there is no real danger of confusing $\supp$ and $\Supp$ at the level of compact objects so we will generally refer to the big support simply as the support and write $\supp$ for both. Of course, the reader may just instead assume that all points of the spectra we consider are visible - this will be the case in our applications.

\subsection{Base change for Rickard idempotents}\label{sec_basechange}
We now prove some facts concerning the behaviour of supports for rigidly-compactly generated tensor triangulated categories under base change along exact monoidal functors. The results are essentially what one would expect in analogy with the case of compact objects. 

We are interested in the following setup: $\mcS$ and $\mcT$ are rigidly-compactly generated tensor triangulated categories and $F\colon \mcS \to \mcT$ is a coproduct and compact object preserving monoidal functor (whose right adjoint, which exists by Brown representability, we denote by $G$ - we note that $G$ necessarily preserves coproducts, see for instance \cite{NeeGrot}*{Theorem~5.1}). We want to consider what $F$ does to the Rickard idempotents of \cite{BaRickard}; our setup is modelled on the situation of base change along a quasi-compact quasi-separated map of schemes. Our results will generalise those in Section~8 of \cite{StevensonActions} to situations more general than localisations.

\begin{notation}
In situations where we wish to distinguish via notation rather than context in which category a tensor product is being taken we will use subscripts e.g., $\otimes_\mcS$ or $\otimes_\mcT$, to be completely clear, and similarly for the support and tensor units.
\end{notation}


Let us denote by $F^c$ the restriction of $F$ to compacts i.e.,
\begin{displaymath}
F^c \colon \mcS^c \to \mcT^c
\end{displaymath}
which exists by assumption. We obtain a spectral (i.e.\ quasi-compact) map of spectral spaces
\begin{displaymath}
\Spc(F^c) = f \colon \Spc \mcT^c \to \Spc \mcS^c
\end{displaymath}
and we know, by \cite{BaSpec}*{Proposition~3.6}, that $\supp Fs = f^{-1}\supp s$ for all $s\in \mcS^c$. In particular, $Fs = 0$ if and only if $f^{-1}\supp s = \varnothing$.

As $f\colon \Spc \mcT^c \to \Spc \mcS^c$ is spectral $f^{-1}$ sends Thomason subsets to Thomason subsets. This leads to the following observation.

\begin{lem}\label{lem_thomsubset_gen}
Let $\mathcal{V}$ be a Thomason subset of $\Spc \mcS^c$. There are equalities of localising ideals of $\mcT$
\begin{displaymath}
\langle F\GGamma_\mathcal{V}\mcS\rangle_\otimes = \langle F\mcS^c_\mathcal{V} \rangle_\otimes = \GGamma_{f^{-1}\mathcal{V}}\mcT.
\end{displaymath}
\end{lem}
\begin{proof}
The first equality is easily checked, for instance it follows from \cite{StevensonActions}*{Lemma~3.8}. We prove the second equality. It is clear from the formula 
\begin{displaymath}
\supp Fs = f^{-1}\supp s
\end{displaymath}
for $s\in \mcS^c$ that $F\mcS^c_\mathcal{V} \cie \mcT^c_{f^{-1}\mathcal{V}}$ and so $\langle F\mcS^c_\mathcal{V} \rangle_\otimes \cie \GGamma_{f^{-1}\mathcal{V}}\mcT$. In fact, one even sees from the support formula that $\supp F\mcS^c_\mathcal{V} = f^{-1}\mathcal{V}$. By Balmer's classification result \cite{BaSpec}*{Theorem~4.10} we thus deduce that the smallest localising tensor ideal containing $F\mcS^c_\mathcal{V}$ contains $\mcT^c_{f^{-1}\mathcal{V}}$ and hence contains $\GGamma_{f^{-1}\mathcal{V}}\mcT$.
\end{proof}

\begin{prop}\label{prop_idempotent_basechange}
Let $\mathcal{V}$ be a Thomason subset of $\Spc \mcS^c$. Then there are natural isomorphisms of Rickard idempotents
\begin{displaymath}
F(\GGamma_\mathcal{V}\mathbf{1}_\mcS) \iso \GGamma_{f^{-1}\mathcal{V}}\mathbf{1}_\mcT \quad \text{and} \quad F(L_\mathcal{V}\mathbf{1}_\mcS) \iso L_{f^{-1}\mathcal{V}}\mathbf{1}_\mcT.
\end{displaymath}
\end{prop}
\begin{proof}
The idempotent $\GGamma_\mathcal{V}\mathbf{1}_\mcS$ comes equipped with a morphism $\GGamma_\mathcal{V}\mathbf{1}_\mcS \stackrel{\varepsilon}{\to} \mathbf{1}_\mcS$ giving the counit of the adjunction corresponding to the acyclisation functor with respect to $\GGamma_\mathcal{V}\mcS$. Applying the monoidal functor $F$ yields
\begin{displaymath}
\varepsilon' = (F(\GGamma_\mathcal{V}\mathbf{1}_\mcS) \stackrel{F(\varepsilon)}{\to} F(\mathbf{1}_\mcS) \stackrel{\sim}{\to} \mathbf{1}_\mcT)
\end{displaymath}
from which we obtain, by tensoring, a natural transformation with component at $X\in \mcT$ $F(\GGamma_\mathcal{V}\mathbf{1}_\mcS)\otimes X \to X$ which we also denote by $\varepsilon'$. We consider the full subcategory $\mathcal{M}$ of $\mcT$ defined as follows
\begin{displaymath}
\mathcal{M} = \{X\in \mcT \; \vert \; \varepsilon'_X \; \text{is an isomorphism} \}.
\end{displaymath}
By naturality of $\varepsilon'$ and its compatibility with coproducts and suspension we deduce immediately that $\mathcal{M}$ is a localising subcategory of $\mcT$. Moreover, given $X\in \mathcal{M}$ and $Y\in \mcT$ commutativity of the square
\begin{displaymath}
\xymatrix{
F(\GGamma_\mathcal{V}\mathbf{1}_\mcS) \otimes (X\otimes Y) \ar[r]^-{\varepsilon'_{X\otimes Y}} \ar[d]^-{\wr} & X\otimes Y \ar@{=}[d] \\
(F(\GGamma_\mathcal{V}\mathbf{1}_\mcS) \otimes X)\otimes Y \ar[r]_-{\varepsilon'_{X}\otimes Y} & X\otimes Y
}
\end{displaymath}
shows that $\mathcal{M}$ is a tensor ideal. Since $F$ is monoidal we have $F(\GGamma_\mathcal{V}\mcS)\cie \mathcal{M}$ and so by the lemma $\GGamma_{f^{-1}\mathcal{V}}\mcT \cie \mathcal{M}$. Thus
\begin{displaymath}
F(\GGamma_\mathcal{V}\mathbf{1}_\mcS) \iso \GGamma_{f^{-1}\mathcal{V}}\mathbf{1}_\mcT \otimes F(\GGamma_\mathcal{V}\mathbf{1}_\mcS) \iso \GGamma_{f^{-1}\mathcal{V}}\mathbf{1}_\mcT
\end{displaymath}
and the corresponding isomorphism for the localisation functors follows from uniqueness of localisation triangles. 
\end{proof}

\begin{cor}
Let $\mathcal{V}$ be a Thomason subset of $\Spc \mcS^c$. Then
\begin{displaymath}
\supp F(\GGamma_\mathcal{V}\mathbf{1}_\mcS) = f^{-1}\mathcal{V} \intersec \Vis \mcT^c \quad \text{and} \quad \supp F(L_\mathcal{V}\mathbf{1}_\mcS) = f^{-1}(\Spc\mcS^c\setminus \mathcal{V}) \intersec \Vis \mcT^c.
\end{displaymath}
\end{cor}

\begin{prop}\label{prop_point_fibre}
Let $y$ be a visible point of $\Spc \mcS^c$. Then 
\begin{displaymath}
\supp F(\GGamma_y\mathbf{1}_\mcS) = f^{-1}(y)\intersec \Vis \mcT^c.
\end{displaymath}
Moreover, $F(\GGamma_y\mathbf{1}_\mcS)$ is zero if and only if the fibre over $y$ is empty.
\end{prop}
\begin{proof}
As $y$ is visible we can find Thomason subsets $\mathcal{V}$ and $\mathcal{W}$ of $\Spc \mcS^c$ defining $\GGamma_y\mathbf{1}_\mcS$ i.e., our subsets satisfy $\mathcal{V}\setminus \mathcal{W} = \{y\}$ and $\GGamma_y\mathbf{1}_\mcS = \GGamma_\mathcal{V}\mathbf{1}_\mcS \otimes L_\mathcal{W}\mathbf{1}_\mcS$. Since $F$ is monoidal we see, using Proposition \ref{prop_idempotent_basechange},
\begin{displaymath}
F(\GGamma_y\mathbf{1}_\mcS) \iso F(\GGamma_\mathcal{V}\mathbf{1}_\mcS) \otimes F(L_\mathcal{W}\mathbf{1}_\mcS) \iso \GGamma_{f^{-1}\mathcal{V}}\mathbf{1}_\mcT \otimes L_{f^{-1}\mathcal{W}}\mathbf{1}_\mcT.
\end{displaymath}
Applying the above Corollary and \cite{StevensonActions}*{Proposition~5.7(4)} then yields the desired equality
\begin{displaymath}
\supp F(\GGamma_y\mathbf{1}_\mcS) 
 = f^{-1}\mathcal{V} \intersec (\Spc\mcT^c \setminus f^{-1}\mathcal{W}) \intersec \Vis \mcT^c = f^{-1}(y) \intersec \Vis\mcT^c.
\end{displaymath}
It follows $F(\GGamma_y\mathbf{1}_\mcS)$ is non-zero iff the fibre over $y$ is non-empty. Indeed, $F(\GGamma_y\mathbf{1}_\mcS)$ is non-zero if and only if 
\begin{displaymath}
\GGamma_{f^{-1}\mathcal{V}}\mathbf{1}_\mcT \otimes L_{f^{-1}\mathcal{W}}\mathbf{1}_\mcT \neq 0 \quad \text{i.e.} \quad f^{-1}\mathcal{V} \nsubseteq f^{-1}\mathcal{W},
\end{displaymath}
which occurs precisely when $f^{-1}(y)\neq \varnothing$.
\end{proof}

\begin{rem}\label{rem_vis_idemp}
Notice that in the special case $f^{-1}(y) = x$ the point $x$ is also visible and we have an isomorphism
\begin{displaymath}
F(\GGamma_y\mathbf{1}_\mcS) \iso \GGamma_x\mathbf{1}_\mcT.
\end{displaymath}
\end{rem}

\subsection{Spectra and smashing localisations}\label{sec_smashing}
We now make some observations concerning the spectra of quotients by smashing ideals. Our point is to illustrate that there is a support-theoretic obstruction to being the right orthogonal of a smashing ideal; the results here form the basis for verifying the telescope conjecture for certain rings in Section~\ref{sec_goodbehaviour}.

Let $\mcT$ be a rigidly-compactly generated tensor triangulated category and $\mcS$ a smashing ideal of $\mcT$. We recall this means there is a localisation sequence
\begin{displaymath}
\xymatrix{
\mcS \ar[r]<0.5ex>^-{I_*} \ar@{<-}[r]<-0.5ex>_-{I^!} & \mcT \ar[r]<0.5ex>^-{J^*} \ar@{<-}[r]<-0.5ex>_-{J_*} & \mcS^\perp
}
\end{displaymath}
where $\mcS$ is a localising tensor ideal of $\mcT$ and $I^!$ preserves coproducts. It follows that $J_*$ also preserves coproducts and $\mcS^\perp$ is also a localising tensor ideal of $\mcT$.

We thus see $\mcS^\perp$ is a compactly generated tensor triangulated category and $J^*$ is an exact monoidal functor which sends compact objects to compact objects. Moreover, there are tensor idempotents $\GGamma_\mcS\mathbf{1}$ and $L_\mcS\mathbf{1}$ such that
\begin{displaymath}
\GGamma_\mcS\mathbf{1}\otimes(-) \iso I_*I^! \quad \text{and} \quad L_\mcS\mathbf{1}\otimes(-) \iso J_*J^*.
\end{displaymath}
Let us denote the restriction of $J^*$ to compacts by $j^*\colon \mcT^c \to (\mcS^\perp)^c$ and the associated spectral map by
\begin{displaymath}
j\colon \Spc (\mcS^\perp)^c \to \Spc \mcT^c.
\end{displaymath}

By \cite{KrCQ}*{Theorem~11.1} the functor $j^*$ is essentially surjective up to direct summands. As an immediate consequence we deduce the following lemma.

\begin{lem}\label{lem_inj}
The map $j\colon \Spc (\mcS^\perp)^c \to \Spc \mcT^c$ is injective.
\end{lem}
\begin{proof}
This is essentially \cite{BaSpec}*{Corollary~3.8} - weakening essential surjectivity to essential surjectivity up to summands does no harm, one just needs to close under summands as well as isomorphisms in the argument given there.
\end{proof}

In order to simplify the discussion, and since it is enough for our intended application, we will assume every point of $\Spc \mcT^c$ is visible. Combining the last lemma with Remark~\ref{rem_vis_idemp} we see every point of $\Spc (\mcS^\perp)^c$ is also visible.

We now identify the image of the map $j$ as a set.

\begin{lem}\label{lem_im_j}
There is an equality of sets
\begin{displaymath}
\im j = \supp \mcS^\perp.
\end{displaymath}
\end{lem}
\begin{proof}
Let $x$ be a point of $\Spc \mcT^c$. We have $x\in \supp \mcS^\perp$ if and only if there is an $X\in \mcS^\perp$ such that $\GGamma_x\mathbf{1}_\mcT \otimes X$ is non-zero. As $Y\iso L_\mcS\mathbf{1}_\mcT \otimes Y$ for all $Y\in \mcS^\perp$ we see there exists such an $X$ if and only if $L_\mcS\mathbf{1}_\mcT\otimes \GGamma_x\mathbf{1}_\mcT$ is non-zero i.e., $J^*\GGamma_x\mathbf{1}_\mcT \neq 0$. We proved in Proposition~\ref{prop_point_fibre} that $J^*\GGamma_x\mathbf{1}_\mcT \neq 0$ if and only if $j^{-1}(x)$ is non-empty.

Tracing through this chain of equivalent statements we find $x\in \supp \mcS^\perp$ if and only if there is a $y\in \Spc (\mcS^\perp)^c$ such that $j(y)=x$, which is precisely what we have claimed.
\end{proof}

\begin{lem}\label{lem_j_homeom}
The map $j$ is closed and is thus a homeomorphism onto its image $\supp \mcS^\perp$ endowed with the subspace topology.
\end{lem}
\begin{proof}
It is, of course, sufficient to check $j$ is closed on a basis of closed subsets for $\Spc (\mcS^\perp)^c$ and so we may reduce to considering the supports of compact objects. Given $a\in (\mcS^\perp)^c$ we will show there is a $b \in \mcT^c$ satisfying
\begin{displaymath}
j(\supp a) = \supp b \intersec \supp \mcS^\perp.
\end{displaymath}
We may assume by replacing $a$, if necessary, by $a\oplus \S a$ that there is a $b\in \mcT^c$ with $j^*b = a$ (using \cite{NeeCat}*{Proposition~4.5.11}); we note making this replacement does not change the support.

We then just need to observe the following series of equalities
\begin{align*}
j(\supp a) &= \{j(\mcP) \; \vert \; \mcP\in \Spc(\mcS^\perp)^c \; \text{and} \; a\notin \mcP\} \\
&= \{(j^*)^{-1}\mcP \; \vert \; \mcP\in \Spc(\mcS^\perp)^c \; \text{and} \; j^*b \notin \mcP\} \\
&= \{(j^*)^{-1}\mcP \; \vert \; \mcP\in \Spc(\mcS^\perp)^c \; \text{and} \; b \notin (j^*)^{-1}\mcP\} \\
&= \{\mcQ \in \supp \mcS^\perp \; \vert \; b\notin \mcQ\} \\
&= \supp b \intersec \supp \mcS^\perp.
\end{align*} 
\end{proof}

In order to state the main result of this section we need to recall the definition of the constructible topology.

\begin{defn}\label{defn_constructible}
Let $X$ be a spectral space. We denote by $X^\mathrm{con}$ the set $X$ equipped with the \emph{constructible topology} which is given by taking the quasi-compact open subsets of $X$ and their complements as a subbasis of open sets.

A subset $Z$ of $X$ which is closed in the constructible topology is called \emph{proconstructible}.
\end{defn}

\begin{rem}
The topology on $X^\mathrm{con}$ is also known as the patch topology, for instance this is the terminology used in \cite{HochsterSpectral}. It is again a spectral space and is Hausdorff.
\end{rem}

\begin{prop}\label{prop_procons}
Let $\mcT$ be a rigidly-compactly generated tensor triangulated category such that every point of $\Spc \mcT^c$ is visible. Given any smashing tensor ideal $\mcS$ of $\mcT$ the subset $\supp \mcS^\perp$ is proconstructible in $\Spc \mcT^c$.
\end{prop}
\begin{proof}
By the last lemma we know $\supp \mcS^\perp$ with the subspace topology is homeomorphic to $\Spc (\mcS^\perp)^c$. Thus $\supp \mcS^\perp$ is a spectral space and the inclusion of $\supp \mcS^\perp$ into $\Spc \mcT^c$ is spectral. The inclusion thus induces a spectral map
\begin{displaymath}
(\supp \mcS^\perp)^\mathrm{con} \to (\Spc \mcT^c)^\mathrm{con}
\end{displaymath}
which is again a homeomorphism onto its image. As $(\Spc \mcT^c)^\mathrm{con}$ is Hausdorff and $(\supp \mcS^\perp)^\mathrm{con}$ is a quasi-compact subset it is closed in $(\Spc \mcT^c)^\mathrm{con}$ i.e., it is proconstructible in $\Spc \mcT^c$.
\end{proof}

\section{Preliminaries on absolutely flat rings}\label{sec_af_prelims}
Before beginning our study of derived categories it seems prudent to provide some brief recollections on the class of rings with which we will be concerned. Throughout all of our rings are assumed to be commutative and unital.

\begin{defn}
Let $R$ be a commutative ring with unit. We say $R$ is \emph{absolutely flat} (also known as \emph{von Neumann regular}) if for every $r\in R$ there exists some $x\in R$ satisfying
\begin{displaymath}
r = r^2 x.
\end{displaymath}
\end{defn}

From now on $R$ denotes a commutative absolutely flat ring. We will assume $R$ is \emph{not} noetherian. 
The following lemma is deduced easily from commutativity of $R$.

\begin{lem}
For every $r\in R$ there is a unique $x\in R$ such that $r = r^2x$ and $x=x^2r$. We call $x$ the \emph{weak inverse} of $r$.
\end{lem}

We now collect some standard characterisations of commutative absolutely flat rings; we will, in general, use these properties without reference to the proposition.

\begin{prop}\label{prop_af_rings}
For a ring $S$ the following are equivalent:
\begin{itemize}
\item[$(i)$] $S$ is absolutely flat;
\item[$(ii)$] $S$ is reduced and has Krull dimension $0$;
\item[$(iii)$] every localization of $S$ at a prime ideal is a field;
\item[$(iv)$] $S$ is a subring of a product of fields, namely
\begin{displaymath}
S \subseteq \prod_{\mathfrak{p}\in \Spec S} k(\mathfrak{p}) =: S',
\end{displaymath}
and $S$ is closed under weak inverses in $S'$;
\item[$(v)$] every simple $S$-module is injective;
\item[$(vi)$] every $S$-module is flat.
\end{itemize}
\end{prop}

The spectrum of $R$, $\Spec R$, is a zero dimensional, Hausdorff, totally disconnected spectral space. We wish to say a little more about it. 
The first two lemmas are trivialities (we note they are valid for the constructible topology on any spectral space).

\begin{lem}\label{lem_closed_qcompact}
A subset $Z\cie \Spec R$ is closed if and only if it quasi-compact.
\end{lem}
\begin{proof}
This is immediate from the fact that $\Spec R$ is Hausdorff.
\end{proof}

\begin{lem}\label{lem_open_thomason}
A subset $V \cie \Spec R$ is Thomason if and only if it is open.
\end{lem}
\begin{proof}
Let $U$ be a quasi-compact open subset of $\Spec R$. Then, by the last lemma, $U$ is also closed and its complement, by virtue of being closed, is also quasi-compact i.e., both $U$ and its complement are closed Thomason subsets. As $\Spec R$ is spectral the subset $V$ is open if and only if it is a union of quasi-compact open subsets of $\Spec R$. But we have just shown this is precisely the same thing as being a union of closed subsets with quasi-compact complement i.e., as being Thomason.
\end{proof}

\begin{lem}\label{lem_spec_inf}
The spectrum of $R$ is infinite i.e., $\vert \Spec R \vert \geq \aleph_0$.
\end{lem}
\begin{proof}
Suppose $\vert \Spec R \vert < \aleph_0$ so $R$ has finitely many prime ideals $\mathfrak{p}_1,\ldots,\mathfrak{p}_n$. Then $R$ is, by Proposition~\ref{prop_af_rings}, a subring of $S= k(\mathfrak{p}_1)\times\ldots\times k(\mathfrak{p}_n)$. Clearly $S$ is noetherian and module finite over $R$ so by the Eakin-Nagata theorem $R$ is noetherian (hence isomorphic to $S$) which is a contradiction.
\end{proof}

\begin{lem}\label{lem_spec_not_noeth}
The spectrum of $R$ is not a noetherian topological space.
\end{lem}
\begin{proof}
By the last lemma $\Spec R$ has infinitely many points. Thus, since it is quasi-compact, $\Spec R$ cannot be discrete. So there is a point $\mathfrak{p}\in \Spec R$ with $\{\mathfrak{p}\}$ not open. Thus $\Spec R \setminus \{\mathfrak{p}\}$ is open but not closed and hence not quasi-compact by Lemma \ref{lem_closed_qcompact}. This demonstrates that $\Spec R$ is not noetherian.
\end{proof}

\begin{rem}
The proof of the lemma exhibits a point of $\Spec R$ which is not a Thomason subset.
\end{rem}


We conclude by reminding the reader that one can functorially associate to any commutative ring $S$ an absolutely flat ring $S^\mathrm{abs}$. Let $\CRing$ denote the category of commutative unital rings and $\CRing^{\mathrm{abs}}$ denote the full subcategory of absolutely flat rings.

\begin{thm}[\cite{Olivier_UAF}*{Proposition~5}]\label{thm_af_adj}
The forgetful functor $\CRing^\mathrm{abs} \to \CRing$ admits a left adjoint
\begin{displaymath}
(-)^\mathrm{abs}\colon \CRing \to \CRing^\mathrm{abs}.
\end{displaymath}
Given a ring $S$ the unit of adjunction $\eta\colon S\to S^\mathrm{abs}$ induces a bijection of sets
\begin{displaymath}
\Spec S^\mathrm{abs} \to \Spec S
\end{displaymath}
and isomorphisms $(S^\mathrm{abs})_P \iso k(\eta^{-1}P)$ for all $P\in \Spec S^\mathrm{abs}$. Furthermore, there is a homeomorphism
\begin{displaymath}
\Spec S^\mathrm{abs} \stackrel{\sim}{\to} (\Spec S)^\mathrm{con},
\end{displaymath}
that is, $S^\mathrm{abs}$ realises the constructible topology on $\Spec S$ (see Definition~\ref{defn_constructible}).
\end{thm}

Given a ring $S$ one can explicitly construct $S^\mathrm{abs}$ as
\begin{displaymath}
S^\mathrm{abs} = S[x_s \; \vert \; s\in S] / (s^2x_s - s, x_s^2s - x_s \; \vert \; s\in S).
\end{displaymath}

\section[Derived categories]{Derived categories of absolutely flat rings}\label{sec_dercat}
Throughout this section $R$ is again a non-noetherian absolutely flat ring. As usual $\D(R)$ denotes the unbounded derived category of $R$ and $\D^\mathrm{perf}(R)$ the full subcategory of perfect complexes. We recall that $\D(R)$ is a rigidly-compactly generated tensor triangulated category with a monoidal model, the subcategory $\D^\mathrm{perf}(R)$ is the full subcategory of compact objects in $\D(R)$, and by a theorem of Thomason \cite{Thomclass}
\begin{displaymath}
\Spc \D^\mathrm{perf}(R) \iso \Spec R.
\end{displaymath}

For $X\in \D(R)$ we set
\begin{displaymath}
\supph X = \{\mathfrak{p}\in \Spec R \; \vert \; X\otimes k(\mathfrak{p}) \neq 0\}.
\end{displaymath}
Since $R$ is absolutely flat there is no need to derive the tensor product although, in any case, $R_\mathfrak{p}\iso k(\mathfrak{p})$. Thus the subset of $\Spec R$ we have defined agrees with the set of primes $\mathfrak{p}$ such that $X_\mathfrak{p}$ is not acyclic. This latter observation yields the following lemma.

\begin{lem}\label{lem_detection}
An object $X$ of $D(R)$ is zero if and only if $\supph X = \varnothing$ i.e., if and only if $X\otimes k(\mathfrak{p}) \iso 0$ for all $\mathfrak{p}\in \Spec R$.
\end{lem}

We now connect this support with Balmer's support on the compact objects and with the corresponding Rickard idempotents. 

\begin{lem}
For every $\mathfrak{p} \in \Spec R$ the subset $\mathcal{Z}(\mathfrak{p}) = \Spec R \setminus \{\mathfrak{p}\}$ is Thomason and the corresponding Rickard idempotent $L_{\mathcal{Z}(\mathfrak{p})}R$ is canonically isomorphic to $k(\mathfrak{p})$. The resulting localisation corresponds to the fully faithful inclusion of $\D(R_\mathfrak{p}) = \D(k(\mathfrak{p}))$ in $\D(R)$.
\end{lem}
\begin{proof}
It is standard that $\mathcal{Z}(\mathfrak{p})$ is Thomason (it is also immediate from Lemma~\ref{lem_open_thomason}) and that the corresponding localisation is given by $\otimes_R R_\mathfrak{p}$ which identifies $R_\mathfrak{p} = k(\mathfrak{p})$ canonically with the corresponding Rickard idempotent.
\end{proof}

Thus, even though there are points which are not Thomason subsets, every point of $\Spc \D^\mathrm{perf}(R)$ is visible. So we can associate to every point $\mathfrak{p}$ a coproduct preserving endofunctor $\GGamma_\mathfrak{p}$. Here $\GGamma_\mathfrak{p}$ is just (yet) another notation for $k(\mathfrak{p})\otimes(-)$
\begin{equation}\label{eq_idempotent}
\GGamma_\mathfrak{p}R = \GGamma_{\Spec R} L_{\mathcal{Z}(\mathfrak{p})}R \iso k(\mathfrak{p}),
\end{equation}
but it is a conceptually helpful one. We see the support theory defined as in \cite{BaRickard}, which we have studied in Section~\ref{sec_nonsense} agrees with the homological support. From this point onward we will just write $\supp X$ for these coinciding notions of the support of an object $X$ of $\D(R)$.

\begin{rem}\label{rem_vis_defn}
We feel it is worth reiterating that this example shows there can exist points whose closure is not Thomason but which are visible and thus have associated idempotents. This shows the definition of visible point, as given in \cite{StevensonActions} (and here in Section~\ref{sec_nonsense}), is more general than the one in \cite{BaRickard}.
\end{rem}

We now investigate $\D(R)$ and, in particular, this support theory with a view toward comparison with rigidly-compactly generated tensor triangulated categories whose compact objects have noetherian spectrum (as considered in \cite{BIKStrat2} and \cite{StevensonActions}). More specifically we consider the analogues of the results of Neeman \cite{NeeChro}, namely the classification of localising subcategories and the telescope conjecture. It turns out that there is a dichotomy - the derived category either behaves very differently (and somewhat mysteriously) or behaves exactly as in the noetherian case where it is possible to completely understand the localising subcategories.

We begin with a few easy general statements; these results are essentially standard but we include most of the details for completeness.

\begin{lem}\label{lem_minimal}
For each $\mathfrak{p}\in \Spec R$ the localising subcategory $\D(k(\mathfrak{p}))$ is minimal in $\D(R)$ i.e., it has no proper non-trivial localising subcategories.
\end{lem}
\begin{proof}
Since $k(\mathfrak{p})$ is a field every object of $\D(k(\mathfrak{p}))$ is a sum of suspensions of $k(\mathfrak{p})$. As localising subcategories are closed under splitting idempotents we see that any non-zero localising subcategory of $\D(k(\mathfrak{p}))$ must be the whole category.
\end{proof}

The next lemma can be found, in a slightly more general form, as \cite{Neeholim}*{Lemma~2.17} and already appears in the work of Foxby \cite{Foxby}.

\begin{lem}\label{lem_res_split}
Let $X$ be an object of $\D(R)$. Then for $\mathfrak{p}\in \Spec R$ the object $X\otimes k(\mathfrak{p})$ is isomorphic to a coproduct of suspensions of $k(\mathfrak{p})$.
\end{lem}

From these two lemmas we deduce the following result.

\begin{lem}\label{lem_min_subcats}
There is a bijection between $\Spec R$ and the minimal non-zero localising subcategories of $\D(R)$ given by sending $\mathfrak{p}$ to the localising subcategory $\D(k(\mathfrak{p}))$ of $\D(R)$.
\end{lem}
\begin{proof}
Lemma~\ref{lem_minimal} associates a unique non-zero minimal localising subcategory to each $\mathfrak{p} \in \Spec R$. It remains to show that these are precisely the non-zero minimal localising subcategories. Suppose then that $\mathcal{L}$ is non-zero, localising and minimal in $\D(R)$. Let $X \in \mathcal{L}$ be a non-zero object. Since, by Lemma~\ref{lem_detection}, the support detects vanishing of objects there is a $\mathfrak{p}\in \Spec R$ with $X\otimes k(\mathfrak{p})$ non-zero. As every localising subcategory of $\D(R)$ is closed under the action, by tensoring, of $\D(R)$ Lemma~\ref{lem_res_split} implies $k(\mathfrak{p}) \in \mathcal{L}$ and hence $\D(k(\mathfrak{p})) \subseteq \mathcal{L}$. Minimality of $\mathcal{L}$ then forces this to be an equality.
\end{proof}



\subsection{Rings behaving badly}\label{sec_bad}

So far this actually seems very promising - the reader might hope that, since we have a seemingly well behaved support theory and know the minimal localising subcategories, we can describe the collection of all localising subcategories in terms of $\Spec R$. The next result should extinguish this hope, at least at our current level of generality (the reader who prefers the good news first should skip to Section \ref{sec_goodbehaviour}).

Let $S$ be a ring (not necessarily absolutely flat). We say $S$ is \emph{semi-artinian} if every non-zero homomorphic image of $S$, in the category of $S$-modules, contains a simple submodule. An $S$-module $M$ is said to be \emph{superdecomposable} if it admits no non-zero indecomposable direct summands. 

\begin{thm}\label{prop_gen_fail}
Let $R$ be an absolutely flat ring which is not semi-artinian. Then the residue fields do not generate $\D(R)$ i.e,
\begin{displaymath}
\langle k(\mathfrak{p})\; \vert \; \mathfrak{p}\in \Spec R\rangle \subsetneq \D(R).
\end{displaymath}
\end{thm}
\begin{proof}
By a result of Trlifaj \cite{TrlifajSuperdecomposable} there exists a superdecomposable injective $R$-module $E$. We claim $E$ is right orthogonal to each $k(\mathfrak{p})$. As $E$ is injective the only other possibility is that $\Hom(k(\mathfrak{p}),E)$ is non-zero. But if there were a non-zero map $k(\mathfrak{p})\to E$ then, since the residue fields are simple, it would have to be a monomorphism. The module $k(\mathfrak{p})$ is injective (see Proposition~\ref{prop_af_rings}), so this morphism would then split and exhibit $k(\mathfrak{p})$ as an indecomposable direct summand of $E$ contradicting the superdecomposability of $E$.
\end{proof}

Combining this with Lemma~\ref{lem_min_subcats} we see, when $R$ is not semi-artinian, the minimal localising subcategories of $\D(R)$ do not generate it. In other words the local-to-global principle (see \cite{StevensonActions}*{Definition~6.1}), with respect to the homological support, fails. This  demonstrates that, in contrast to the case where the spectrum is noetherian (\cite{StevensonActions}*{Theorem~6.9}), the support detecting vanishing is not equivalent to the local-to-global principle in general. 


\begin{defn}
A localising subcategory $\mathcal{L}$ of $\D(R)$ is a \emph{Bousfield class} if there is an $X\in \D(R)$ such that
\begin{displaymath}
\mathcal{L} = \{Y\in \D(R) \; \vert \; X\otimes Y \iso 0\} = \ker(X\otimes(-)).
\end{displaymath}
The localising subcategory $\mathcal{L}$ is a \emph{cohomological Bousfield class} if there is an $X\in \D(R)$ such that
\begin{displaymath}
\mathcal{L} = \{Y\in \D(R) \; \vert \; \Hom(Y,\S^iX) = 0 \;\; \forall i\in \int\} = {}^\perp X.
\end{displaymath}
We call $\ker(X\otimes(-))$ the Bousfield class of $X$ and ${}^\perp X$ the cohomological Bousfield class of $X$.
\end{defn}

\begin{cor}\label{cor_not_bousfield}
If $R$ is not semi-artinian then $\D(R)$ admits a localising subcategory which is not a Bousfield class.
\end{cor}
\begin{proof}
By the theorem $\mathcal{L} = \langle k(\mathfrak{p})\;\vert \; \mathfrak{p}\in \Spec R\rangle$ is a proper localising subcategory of $\D(R)$. In particular it is not the Bousfield class of $0$. As the support detects vanishing the Bousfield class of any non-zero $X\in \D(R)$ fails to contain some $k(\mathfrak{p})$. Thus $\mathcal{L}$ cannot be a Bousfield class.
\end{proof}

\begin{rem}
This gives a counterexample to the analogue of \cite{HP_Bousfield}*{Conjecture~9.1} for the derived category of a ring. 
\end{rem}

\begin{cor}
If $R$ is not semi-artinian then $\D(R)$ admits a cohomological Bousfield class which is not a Bousfield class.
\end{cor}
\begin{proof}
As in the proof of the theorem there is, by \cite{TrlifajSuperdecomposable}, a superdecomposable injective $E$ and
\begin{displaymath}
\langle k(\mathfrak{p}) \; \vert \; \mathfrak{p} \in \Spec R \rangle \cie {}^\perp E.
\end{displaymath}
It is clear that ${}^\perp E$ is a proper localising subcategory of $\D(R)$ and so, by Lemma~\ref{lem_detection}, ${}^\perp E$ cannot be a Bousfield class as in the argument proving the last Corollary.
\end{proof}

\begin{rem}
This gives a counterexample to the analogue of \cite{Hovey_CBC}*{Conjecture~1.2} for the derived category of a ring.
\end{rem}

\begin{ex}
Let $\Lambda$ be an infinite index set and for each $\l\in \Lambda$ let $k_\l$ be a field. The ring
\begin{displaymath}
R = \prod_{\l \in \Lambda} k_\l
\end{displaymath}
is absolutely flat. By \cite{ClarkSmith}*{Lemma~1} it is not semi-artinian and so the theorem and both corollaries apply. In particular, the non-zero minimal localising subcategories of $\D(R)$ (which correspond to ultrafilters on $\Lambda$) do not generate $\D(R)$ and the local-to-global principle for the action of $\D(R)$ on itself fails.
\end{ex}




\subsection{Some good behaviour}\label{sec_goodbehaviour}
We now use the abstract nonsense of Section~\ref{sec_nonsense} to show that for a certain class of absolutely flat rings one does not observe the same interesting behaviour as in the last section. 

Let $S$ be a commutative unital ring such that $\Spec S$ is a noetherian topological space, for instance $S$ could be a noetherian ring. Since $\Spec S$ is noetherian every point of $\Spc D^\mathrm{perf}(S) \iso \Spec S$ is visible. Set $R= S^\mathrm{abs}$ and denote by $f\colon S\to R$ the canonical map (see Theorem~\ref{thm_af_adj} for details). Our first aim is to prove that $\Spec S$ being noetherian implies $\D(R)$ is generated by the residue fields of $R$.

The first lemma we need is just an application of our general nonsense to the exact, monoidal, and compact object preserving functor $\L f^*\colon \D(S) \to \D(R)$.

\begin{lem}\label{lem_idempotent_computation}
Let $\mathfrak{p}\in \Spec S$ and denote by $P$ the unique point of $\Spec R$ such that $f^{-1}P = \mathfrak{p}$. Then
\begin{displaymath}
\L f^*(\GGamma_\mathfrak{p}S) \iso k(P).
\end{displaymath}
\end{lem}
\begin{proof}
As in the statement let $\mathfrak{p}$ be a point of $\Spec S$ and $P$ be the corresponding point of $\Spec R$.  We are in the situation of Section~\ref{sec_nonsense} so we may apply Proposition~\ref{prop_idempotent_basechange} to obtain isomorphisms,
\begin{align*}
\L f^*(\GGamma_\mathfrak{p}S) &= \L f^*(\GGamma_{\mathcal{V}(\mathfrak{p})}S \otimes L_{\mathcal{Z}(\mathfrak{p})}S)\\
& \iso \L f^*(\GGamma_{\mathcal{V}(\mathfrak{p})}S) \otimes \L f^*(L_{\mathcal{Z}(\mathfrak{p})}S)\\
&\iso \GGamma_{f^{-1}\mathcal{V}(\mathfrak{p})}R \otimes L_{f^{-1}\mathcal{Z}(\mathfrak{p})}R.
\end{align*}
As $f$ is a bijection we have 
\begin{displaymath}
f^{-1}\mathcal{V}(\mathfrak{p}) \setminus f^{-1}\mathcal{Z}(\mathfrak{p}) = \{P\}.
\end{displaymath}
This yields isomorphisms
\begin{displaymath}
\GGamma_{f^{-1}\mathcal{V}(\mathfrak{p})}R \otimes L_{f^{-1}\mathcal{Z}(\mathfrak{p})}R \iso \GGamma_P R \iso k(P),
\end{displaymath}
the first by the independence of $\GGamma_P R$ on the subsets used to define it (see \cite{BaRickard}*{Corollary~7.5}) and the second by the observation of equation~\ref{eq_idempotent}, completing the proof.
\end{proof}

\begin{rem}
Assuming the spectrum of $S$ is noetherian is not essential for the lemma provided one corrects the statement by considering only visible points (i.e., the points for which the statement makes sense).
\end{rem}


Now we invoke the hypothesis that $\Spec S$ is noetherian. As we have mentioned earlier $\D(S)$ is always a rigidly-compactly generated tensor triangulated category with a monoidal model and, by a result of Thomason \cite{Thomclass}, $\Spc \D^{\mathrm{perf}}(S)$ is canonically homeomorphic to $\Spec S$. Thus, as $\Spec S$ is noetherian, the hypotheses of \cite{StevensonActions}*{Theorem~6.9} apply: the local-to-global principle holds for the action of $\D(S)$ on itself and the support detects vanishing. In particular we have equalities
\begin{equation}\label{eq_generation}
\D(S) = \langle S \rangle = \langle \GGamma_\mathfrak{p}S \; \vert \; \mathfrak{p}\in \Spec S\rangle.
\end{equation}
Combined with the lemma this has the following consequence.

\begin{prop}\label{prop_gen_win}
The derived category $\D(R)$ is generated by the residue fields of $R$ i.e.,
\begin{displaymath}
\D(R) = \langle k(P) \; \vert \; P\in \Spec R\rangle.
\end{displaymath}
\end{prop}
\begin{proof}
By (\ref{eq_generation}) $S$ can be built from the $\GGamma_\mathfrak{p}S$ so by a standard argument
\begin{displaymath}
R \iso \L f^* S \in \langle \L f^*(\GGamma_\mathfrak{p}S) \; \vert \; \mathfrak{p}\in \Spec S \rangle.
\end{displaymath}
Applying the last lemma this says $R$ lies in the localising subcategory $\langle k(P) \; \vert \; P\in \Spec R \rangle$. This proves the proposition as any localising subcategory of $\D(R)$ containing $R$ must be all of $\D(R)$.
\end{proof}

\begin{cor}
If $S$ is a ring with noetherian spectrum then $R= S^\mathrm{abs}$ is semi-artinian.
\end{cor}
\begin{proof}
We have proved in Theorem~\ref{prop_gen_fail} that if $R$ is not semi-artinian the residue fields do not generate $\D(R)$, so, given the theorem, $R$ had better be semi-artinian.
\end{proof}

\begin{rem}
One can also prove the corollary using more typical methods. If $S$ has noetherian spectrum then a straightforward computation shows the Cantor-Bendixson rank of $\Spec S^\mathrm{abs}$ is at most $\omega$.
\end{rem}

In fact we could have proved the proposition by more general methods. The monoidal functor $\L f^*$ furnishes us with an action of $\D(S)$ on $\D(R)$ as defined in \cite{StevensonActions}. As we have noted above \cite{StevensonActions}*{Theorem~6.9} guarantees this action satisfies the local-to-global principle. Furthermore, we have shown in Lemma \ref{lem_idempotent_computation} that for $\mathfrak{p}\in \Spec S \iso \Spc \D^\mathrm{perf}(S)$ the object $\GGamma_\mathfrak{p}S$ acts on $\D(R)$ as $k(P)$ where $P$ is the corresponding point of $\Spec R$. Thus the support obtained from the action of $\D(S)$ on $\D(R)$ agrees with the homological support on $\D(R)$. This essentially proves the following lemma.

\begin{lem}
The homological support on $\D(R)$ satisfies the local-to-global principle i.e., for $A\in \D(R)$
\begin{displaymath}
\langle A \rangle = \langle k(P) \; \vert \; P\in \supp A\rangle.
\end{displaymath}
\end{lem}

We thus deduce a classification theorem for localising subcategories of $\D(R)$.

\begin{thm}\label{thm_class}
There is an order preserving bijection
\begin{displaymath}
\left\{ \begin{array}{c}
\text{subsets of}\; \Spec R 
\end{array} \right\}
\xymatrix{ \ar[r]<1ex>^{\tau} \ar@{<-}[r]<-1ex>_{\sigma} &} \left\{
\begin{array}{c}
\text{localising subcategories of} \; \D(R) \\
\end{array} \right\}, 
\end{displaymath}
where for a localising subcategory $\mathcal{L}$ and a subset $W\cie \Spec R$ we set
\begin{displaymath}
\sigma(\mathcal{L}) = \{ P\in \Spec R \; \vert \; k(P)\otimes \mathcal{L}\neq 0\} \quad \text{and} \quad \tau(W) = \langle k(P) \; \vert \; P\in W\rangle.
\end{displaymath}
\end{thm}
\begin{proof}
The map $\tau$ is a split monomorphism with left inverse $\sigma$ by \cite{StevensonActions}*{Proposition~6.3}. Thus to prove the result we just need to observe that
\begin{align*}
\tau\sigma(\mathcal{L}) &= \tau(\{P \in \Spec R \; \vert \; k(P)\otimes \mathcal{L} \neq 0\})\\
&= \langle k(P) \; \vert \; k(P) \otimes \mathcal{L} \neq 0\rangle
\,
\end{align*}
which is precisely $\mathcal{L}$: $\tau\sigma(\mathcal{L}) \cie \mathcal{L}$ is immediate as every localising subcategory of $\D(R)$ is a $\otimes$-ideal and $\mathcal{L} \cie \tau\sigma(\mathcal{L})$ is a straightforward consequence of the last lemma.
\end{proof}

Thus Neeman's classification \cite{NeeChro} extends to the absolutely flat approximations of rings with noetherian spectrum. Given this classification it is natural to ask if we can settle the telescope conjecture in this setting. It turns out to be an easy consequence of the formal results we have proved in Section~\ref{sec_smashing}.

\begin{thm}\label{thm_telescope}
The telescope conjecture holds for $\D(R)$ i.e., every smashing subcategory of $\D(R)$ is generated by compact objects.
\end{thm}
\begin{proof}
Suppose $\mcS$ is a smashing subcategory of $\D(R)$ and let $\mcS^\perp$ be its right orthogonal, which is also a localising subcategory. By the theorem $\mcS$ and $\mcS^\perp$ are determined by the subsets $\sigma(\mcS)$ and $\sigma(\mcS^\perp)$ of $\Spec R$. One can easily check
\begin{displaymath}
\sigma \mcS = \Spec R \setminus \sigma \mcS^\perp
\end{displaymath}
(cf.\ \cite{StevensonActions}*{Lemma~7.13}). By Proposition~\ref{prop_procons} the subset $\sigma \mcS^\perp$ is proconstructible in $\Spec R$ and hence is closed as $\Spec R$ is already equipped with the constructible topology. Thus $\sigma \mcS$ is open and hence Thomason by Lemma~\ref{lem_open_thomason}. This shows $\mcS$ is compactly generated as, again using the theorem, we have
\begin{displaymath}
\mcS = \tau(\sigma \mcS) = \GGamma_{\sigma \mcS}\D(R) = \langle \D^\mathrm{perf}_{\sigma \mcS}(R)\rangle.
\end{displaymath}
\end{proof}


\subsection{Absolutely flat schemes}\label{sec_schemes}
We indicate here how to extend the results we have obtained in Section~\ref{sec_goodbehaviour} to the analogous class of schemes. This involves no extra work as the formalism of tensor actions allows us to deduce global results affine locally.

Given a scheme $X$ one can globalise Theorem~\ref{thm_af_adj} (see \cite{Olivier_le_foncteur}) to obtain a universal map of schemes $X^\mathrm{abs} \to X$ where $X^\mathrm{abs}$ is an absolutely flat scheme i.e., $X^\mathrm{abs}$ admits an open affine cover by the spectra of absolutely flat rings. For any open affine subscheme $\Spec S$ of $X$ its preimage in $X^\mathrm{abs}$ is just $\Spec S^\mathrm{abs}$.

\begin{thm}
Let $X$ be a topologically noetherian scheme. The action of $\D(X)$ (or $\D(X^\mathrm{abs})$) on $\D(X^\mathrm{abs})$ gives an order preserving bijection
\begin{displaymath}
\left\{ \begin{array}{c}
\text{subsets of}\; X^\mathrm{abs} 
\end{array} \right\}
\xymatrix{ \ar[r]<1ex>^{\tau} \ar@{<-}[r]<-1ex>_{\sigma} &} \left\{
\begin{array}{c}
\text{localising ideals of} \; \D(X^\mathrm{abs}) \\
\end{array} \right\}, 
\end{displaymath}
where for a localising ideal $\mathcal{L}$ and a subset $W\cie X^\mathrm{abs}$ we set
\begin{displaymath}
\sigma(\mathcal{L}) = \{ x\in X^\mathrm{abs} \; \vert \; k(x)\otimes \mathcal{L}\neq 0\} \quad \text{and} \quad \tau(W) = \langle k(x) \; \vert \; x\in W\rangle_\otimes.
\end{displaymath}
\end{thm}
\begin{proof}
Given Theorem~\ref{thm_class} one just chooses an open affine cover of $X^\mathrm{abs}$ and applies \cite{StevensonActions}*{Theorem~8.11}.
\end{proof}

\begin{thm}
Let $X$ be a topologically noetherian scheme. The relative telescope conjecture holds for the action of $\D(X^\mathrm{abs})$ on itself i.e., every smashing tensor ideal of $\D(X^\mathrm{abs})$ is generated by objects of $\D^\mathrm{perf}(X^\mathrm{abs})$.
\end{thm}
\begin{proof}
Again this follows from the result in the affine case, namely Theorem~\ref{thm_telescope}. One chooses an open affine cover for $X^\mathrm{abs}$ and applies \cite{BaRickard}*{Theorem~6.6}.
\end{proof}

  \bibliography{greg_bib}

\end{document}